\documentclass[12pt, reqno]{amsart}
\usepackage{subfigure}
\usepackage{graphics,epstopdf}
\usepackage[pdftex]{graphicx}
\usepackage{amsmath, amsthm, amscd, amsfonts, amssymb, graphicx, color}
\usepackage[bookmarksnumbered, colorlinks, plainpages]{hyperref}
\hypersetup{colorlinks=true,linkcolor=red, anchorcolor=green,
citecolor=cyan, urlcolor=red, filecolor=magenta, pdftoolbar=true}

\newtheorem{thm}{Theorem}
\numberwithin{equation}{section}

\newtheorem{lem}[thm]{Lemma}

\newtheorem{definition}{Definition}[section]

\numberwithin{equation}{section}

\begin{document}
\setcounter{page}{1}
\title[Approximation by Meyer-Konig and Zeller operators using $(p,q)$-calculus ]{Approximation by Meyer-Konig and Zeller operators using $(p,q)$-calculus}

\author[U. Kadak, Asif khan, M. Mursaleen, ]{U. Kadak$^{1}$, Asif khan $^{2}$  and  M. Mursaleen $^3$ }

\address{$^1$Department of Mathematics, Bozok University, 66100 Yozgat , Turkey \\
\newline$^2,^3$ Department of Mathematics, Aligarh Muslim University,  India}
\email{ugurkadak@gmail.com (U. Kadak), {\tt asifjnu07@gmail.com}(A. Khan), \newline {\tt mursaleenm@gmail.com}(M. Mursaleen)}.



\begin{abstract}
In this paper, we introduce a generalization of the $q$-Meyer-Konig and Zeller operators by means of the $(p,q)$-integers as well as of the $(p,q)$-Gaussian binomial coefficients. For $0 < q <p \leq 1,$  we denote the sequence of the $(p,q)$-Meyer-Konig and
Zeller operators by $M_{n,p, q}$ and obtain some direct theorems and results based on statistical convergence.

 Furthermore, we show comparisons and some illustrative graphics for the convergence of operators to a function.
\end{abstract} \maketitle

\section{Introduction and preliminaries}
In 1960, starting from the identity
\begin{eqnarray}\label{43}
(1-x)^{n+1}\sum_{k=0}^{\infty}\binom{n+k}{k}x^k=1~~\text{for all}~~x\in [0, 1),
\end{eqnarray}
Meyer-Konig and Zeller \cite{meyer} introduced a sequence  of linear positive operators by using real continuous functions defined on $[0, 1)$. In the slight modification by Cheney and Sharma \cite{cheney}, the Meyer-Konig and Zeller operators are defined on $C[0, 1]$  by
\begin{eqnarray}\label{47}
M_n(f; x)&=&\sum_{k=0}^{\infty}f\left(\frac{k}{n+k}\right)\binom{n+k}{k}x^k(1-x)^{n+1}~~~\text{if}~~x\in [0, 1),\\
\nonumber M_n(f; 1)&=&f(1) ~~~\text{if}~~x=1,~~\text{for every}~n \in \mathbb N.
\end{eqnarray}

Based on the development of $q$-calculus \cite{rahman,vp} many authors studied some new generalizations of linear positive operators based on $q$-integers \cite{aral,ms314,hp,sofia,pl,ph1}.
Very first $q$-analogue of Bernstein operators \cite{brn} was introduced and investigated corresponding approximating properties by Lupa\c s \cite{lupas}. The linear operators $L_n,q: C[0, 1] \to  C[0, 1]$, defined by
\begin{eqnarray*}
L_n,q(f;x)=\sum_{k=0}^{n}\frac{f\left(\frac{[k]_q}{[n]_q}\right){n \brack k}_{q}q^\frac{k(k-1)}{2}x^k (1-x)^{n-k}}{\prod_{j=1}^{n}\{(1-x)+q^{j-1}(x)\}}
\end{eqnarray*} are known as Lupa\c s $q$-analogue of Bernstein operators.\\

Later on, Phillips \cite{pl} proposed another $q$-variant of the classical Bernstein operator, the so-called Phillips $q$-Bernstein operators, as

\begin{equation}
B_{n,q}(f;x)=\sum\limits_{k=0}^{n}\left[
\begin{array}{c}
n \\
k%
\end{array}%
\right] _{q}x^{k}\prod\limits_{s=0}^{n-k-1}(1-q^{s}x)~~f\left( \frac{%
[k]_{q}}{[n]_{q}}\right) ,~~x\in \lbrack 0,1]
\end{equation}
where $B_{n,q}: $ $C[0,1]\rightarrow C[0,1]$ defined for any $n\in \mathbb{N}$
and any function $f\in C[0,1].$ \\

In a recent paper \cite{trif}, Tiberiu Trif  introduced the $q$-Meyer-Konig and Zeller operators for $n\in \mathbb N$ and $f\in C[0, 1]$,
\begin{eqnarray}\label{60}
M_{n,q}(f, x)&=&\sum_{k=0}^{\infty}f\left(\frac{[k]_q}{[n+k]_q}\right){n+k \brack k}_q~x^k(1-x)_q^{n+1}~~\text{if}~~x\in [0, 1),\\
\nonumber M_{n,q}(f, 1)&=&f(1),~~\text{if}~~x=1.
\end{eqnarray}

For $q=1,$ the operator $M_{n,q}$ is reduced to the classical Meyer-Konig and Zeller operator given in (\ref{47}). Furthermore, in the slight modification by Dogru and Duman in \cite{dd} the $q$-Meyer-Konig and Zeller operators are defined on $C[0, a]$, $a\in (0, 1)$,  by
\begin{eqnarray*}\label{43}
M_{n,q}(f, x)&=&\prod_{s=0}^{n}(1-xq^s)\sum_{k=0}^{\infty}f\left(\frac{q^n[k]_q}{[n+k]_q}\right){n+k \brack k}_q~x^k,~~~(q\in(0, 1]),~~ n \in \mathbb N.
\end{eqnarray*}

For some other works on Meyer-Konig and Zeller operator, one can see \cite{abel,ms314}.\\

Mursaleen et al. \cite{mka1}  first applied the concept of $(p,q)$-calculus in approximation theory and introduced the $(p,q)$-analogue of Bernstein operators. Later, based on $(p,q)$-integers, some approximation results for  Bernstein-Stancu operators, Bernstein-Kantorovich operators, Lorentz operators, Bleimann-Butzer and Hahn operators and Bernstein-Shurer operators etc.  have also been introduced in \cite{mur8,mka3,mka5,mnak1,zmn}.

Motivated by the work of Mursaleen et al \cite{mka1}, the idea of $(p,q)$-calculus and its importance.\\

Very recently, Khalid et al. \cite{khalid1,khalid2,khalid3} have given a nice application in computer-aided geometric design and applied these Bernstein basis  for construction of $(p,q)$-B$\acute{e}$zier curves and surfaces based on  $(p,q)$-integers which is further generalization of $q$-B$\acute{e}$zier curves and surfaces \cite{bezier,hp,pl,ph1}. For similar works based on $(p,q)$-integers, one can refer \cite{acar1,acar3,cai,mah,ali,ali1,ma1,wafi}.\\


Motivated by the above mentioned work on $(p,q)$-approximation and its application, in this paper we construct a Meyer-Konig and Zeller operators based on $(p,q)$-integers as an extension of
$q$-Meyer-Konig and Zeller operators \cite{trif} and study some approximation properties.

Now we recall some basic definitions about $(p,q)$-integers. For any $n\in \mathbb N$, the $(p,q)$-integer $[n]_{p,q}$ is defined by
$$[0]_{p, q}:=0~~\text{and}~~[n]_{p, q}=\frac{p^n-q^n}{p-q}~~\text{if}~~n\ge 1,$$ where $0<q<p\le 1$. The $(p,q)$-factorial is defined by
$$[0]_{p, q}!:=1~~\text{and}~~[n]!_{p, q}=[1]_{p, q}[2]_{p, q}\cdots [n]_{p, q}~~\text{if}~~n\ge 1.$$
Also the $(p,q)$-binomial coefficient is defined by
$${n \brack k}_{p, q}=\frac{[n]_{p, q}!}{[k]_{p, q}!~[n-k]_{p, q}!}~~\text{for all}~~n, k\in \mathbb N~~\text{with}~~n\ge k.$$

The formula for $(p,q)$-binomial expansion is as follows:
\begin{equation*}
(ax+by)_{p,q}^{n}:=\sum\limits_{k=0}^{n}p^{\frac{(n-k)(n-k-1)}{2}}q^{\frac{k(k-1)}{2}}
{n \brack k}_{p, q}a^{n-k}b^{k}x^{n-k}y^{k},
\end{equation*}
$$(x+y)_{p,q}^{n}=(x+y)(px+qy)(p^2x+q^2y)\cdots (p^{n-1}x+q^{n-1}y),$$
$$(1-x)_{p,q}^{n}=(1-x)(p-qx)(p^2-q^2x)\cdots (p^{n-1}-q^{n-1}x),$$\\
%

Details on $(p,q)$-calculus can be found in \cite{mah,jag,mka1}.\\

 The $(p,q)$-Bernstein operators introduced by Mursaleen et al. are defined as \cite{mka1} follows:
\begin{equation}\label{e1}
B_{n,p,q}(f;x)=\frac1{p^{\frac{n(n-1)}2}}\sum\limits_{k=0}^{n}\left[
\begin{array}{c}
n \\
k%
\end{array}%
\right] _{p,q}p^{\frac{k(k-1)}2}x^{k}\prod\limits_{s=0}^{n-k-1}(p^{s}-q^{s}x)~~f\left( \frac{%
[k]_{p,q}}{p^{k-n}[n]_{p,q}}\right) ,~~x\in \lbrack 0,1].
\end{equation}

Also, we have
\begin{align*}
(1-x)^{n}_{p,q}&=\prod\limits_{s=0}^{n-1}(p^s-q^{s}x) =(1-x)(p-qx)(p^{2}-q^{2}x)...(p^{n-1}-q^{n-1}x)\\
&=\sum\limits_{k=0}^{n} {(-1)}^{k}p^{\frac{(n-k)(n-k-1)}{2}} q^{\frac{k(k-1)}{2}}\left[
\begin{array}{c}
n \\
k
\end{array}%
\right] _{p,q}x^{k}
\end{align*}

Again by induction and using the property of $(p,q)$-integers, we get $(p,q)$-analogue of Pascal's relation \cite{khalid1,khalid2} as follows:

\begin{equation}\label{e2}
\left[
\begin{array}{c}
n \\
k%
\end{array}%
\right] _{p,q}= q^{n-k}\left[
\begin{array}{c}
n-1 \\
k-1%
\end{array}%
\right] _{p,q}+ p^{k}\left[
\begin{array}{c}
n-1 \\
k%
\end{array}%
\right] _{p,q}
\end{equation}

\begin{equation}\label{e3}
\left[
\begin{array}{c}
n \\
k%
\end{array}%
\right] _{p,q}= p^{n-k}\left[
\begin{array}{c}
n-1 \\
k-1%
\end{array}%
\right] _{p,q}+ q^{k}\left[
\begin{array}{c}
n-1 \\
k%
\end{array}%
\right] _{p,q}
\end{equation}

\section{Construction of operators}

In this section, we first introduce $(p,q)$-analogue of Meyer-Konig and Zeller operators and present three auxillary lemmas for the proposed operators.

With the help of mathematical induction on $n,$ it can be easily verified that\\

\begin{eqnarray}\label{ee10}
\frac{1}{\prod_{s=0}^{n}(p^s-q^sx)}=\frac{1}{p^\frac{n(n+1)}{2}}\sum_{k=0}^{\infty}{n+k \brack k}_{p, q} x^kp^{-kn}
\end{eqnarray}

Let $f$ be a function defined on the interval $[0, 1]$ and $0<q< p\le 1$. We, define the following Meyer-Konig and Zeller operators based on $(p,q)$-integers as follows:

\begin{eqnarray}\label{87}
\nonumber M_{n,p, q}(f; x)&=&\frac{1}{p^\frac{n(n+1)}{2}}\sum_{k=0}^{\infty}{n+k \brack k}_{p, q} x^k p^{-kn}~\prod_{s=0}^{n}(p^s-q^sx)f\left(\frac{p^{n}[k]_{p, q}}{[n+k]_{p, q}}\right)\text{if}~~x\in [0, 1),\\
M_{n, p, q}(f, 1)&=&f(1),~~\text{if}~~x=1,
\end{eqnarray}
and we call as $(p,q)$-analogue of MKZ operators, ($(p,q)$-MKZ operators). Also taking $p=1$ in (\ref{87}), we immediately get the $q$-MKZ operator $M_{n,q}$ defined by (\ref{60}).

\begin{lem} \label{l1} For all $x \in [0, 1]$ and $0<q<p\le 1,$ we have
\begin{enumerate}
\item [(i)] $M_{n,p, q}(1; x)=1;$\\
\item [(ii)] $M_{n,p, q}(t; x)=x;$\\
\item [(iii)] $ x^2 \leq M_{n,p, q}(t^2; x)\leq \frac{p^n}{[n+1]_{p, q}} x+x^2$.
\end{enumerate}
\end{lem}

\begin{proof}
\begin{itemize}
\item [(i)]
\begin{eqnarray}\label{ee1}
M_{n,p, q}(1; x)=\frac{1}{p^\frac{n(n+1)}{2}}\sum_{k=0}^{\infty}{n+k \brack k}_{p, q} x^kp^{-kn}~\prod_{s=0}^{n}(p^s-q^sx)=1
\end{eqnarray}
\end{itemize}
\end{proof}

\begin{lem}\label{l2}
 For all $x \in [0, 1]$ and $0<q<p\le 1,$ we have
\begin{enumerate}
\item [(i)] $M_{n,p, q}((t-x); x)=0;$\\
\item [(ii)] $M_{n,p, q}((t-x)^2; x)\leq \frac{p^n}{[n+1]_{p, q}} x+(p-1)x^2$.
\end{enumerate}
\end{lem}

\begin{proof}
\begin{itemize}
\item [(ii)] Taking into account the linearity of the operators we get
\begin{eqnarray*}
M_{n,p, q}((t-x)^2; x)&=&M_{n,p, q}(t^2; x)-2x M_{n,p, q}(t; x)+ x^2M_{n,p, q}(1; x)\\
&\leq&\frac{p^n}{[n+1]_{p, q}} x+px^2-2x^2+x^2\\
&=&\frac{p^n}{[n+1]_{p, q}} x+(p-1)x^2.
\end{eqnarray*}
\end{itemize}
\end{proof}

\section{ Main Results}

\subsection{Korovkin type approximation theorem}

We know that $C[a,b]$ is a
Banach space with norm
\begin{equation*}
\Vert f\Vert _{C[a,b]}:=\sup\limits_{x\in \lbrack a,b]}|f(x)|,~f\in C[a,b].
\end{equation*}
For typographical convenience, we will write $\Vert .\Vert $ in place of $\Vert .\Vert _{C[a,b]}$ if no confusion arises.\newline

\begin{definition} {\em Let $C[a,b]$ be the linear space of all real valued continuous functions $f$
on $[a,b]$ and let $T$ be a linear operator which maps $C[a,b]$ into itself.
We say that $T$ is $positive$ if for every non-negative $f\in $ $C[a,b],$ we
have $T(f,x)\geq 0$ for all $x\in $ $[a,b]$ .}
\end{definition}

\parindent=8mm The classical Korovkin type approximation theorem can be stated as follows \cite{korovkin};

Let $T_n: C[a, b] \to C[a, b]$ be a sequence of positive linear operators.
Then $\lim_{n\to\infty}\|T_{n}(f; x)-f(x)\|_\infty=0,\,\,\textrm{for
all}~f\in C[a, b]$ if and only if $\lim_{n\to\infty}\|T_{n}(f_{i}; x)-f_i(x)\|_\infty=0,\,\,\textrm{for each}\,\,i=0,1,2,$ where the test function $f_i(x)=x^i$.

\begin{thm}
 Let $0<q_n<p_n\leq1$ such that $\lim\limits_{n\to\infty}p_n=1$ and $\lim\limits_{n\to\infty}q_n=1$. Then for each $f\in C[0,1],~M_{n,p_n,q_n}(f;x)$ converges uniformly to $f$ on $[0,1]$.
\end{thm}

\parindent=0mm\textbf{Proof}. By the Bohman-Korovkin Theorem it is sufficient to show that
\begin{equation*}
\lim\limits_{n\to\infty}\|M_{n,p_n,q_n}(t^m;x)-x^m\|_{C[0,1]}=0,~~~m=0,1,2.
\end{equation*}
By Lemma\ref{l1}(i)-(ii), it is clear that
\begin{equation*}
\lim\limits_{n\to\infty}\|M_{n,p_n,q_n}(1;x)-1\|_{C[0,1]}=0;
\end{equation*}
\begin{equation*}
\lim\limits_{n\to\infty}\|M_{n,p_n,q_n}(t;x)-x\|_{C[0,1]}=0;
\end{equation*}
and by Lemma \ref{l1}(iii), we have
\begin{eqnarray*}
|M_{n,p_n,q_n}(t^2;x)-x^2|_{C[0,1]} \leq \| \frac{p^n}{{[n+1]}_{p_n,q_n}}x\|
\end{eqnarray*}

Taking maximum of both sides of the above inequality, we get\\
\begin{equation*}
\|M_{n,p_n,q_n}(t^2;x)-x^2\|_{C[0,1]}\leq\frac{p_n^{n}}{[n+1]_{p_n,q_n}}
\end{equation*}
which yields
\begin{equation*}
\lim\limits_{n\to\infty}\|M_{n,p_n,q_n}(t^2;x)-x^2\|_{C[0,1]}=0.
\end{equation*}
Thus the proof is completed.\\

\subsection{Direct Theorems}

\parindent=8mm The Peetre's $K$-functional is defined by \newline
\begin{equation*}
K_{2}(f,\delta )=\inf [\{\Vert f-g\Vert +\delta \Vert g^{\prime \prime
}\Vert \}:g\in W^{2}],
\end{equation*}%
where%
\begin{equation*}
W^{2}=\{g\in C[0,1]:g^{\prime },g^{\prime \prime }\in
C[0,1]\}.
\end{equation*}

By \cite{lb313}, there exists a positive constant $C>0$ such that $%
K_{2}(f,\delta )\leq Cw_{2}(f,\delta ^{\frac{1}{2}}),~~\delta >0;$ where the
second order modulus of continuity is given by
\begin{equation*}
w_{2}(f,\delta ^{\frac{1}{2}})=\sup\limits_{0<h\leq \delta ^{\frac{1}{2}%
}}\sup\limits_{x\in \lbrack 0,1]}\mid f(x+2h)-2f(x+h)+f(x)\mid .
\end{equation*}%
\newline

Also for $f\in [0,1]$ the usual modulus of continuity is given by
\begin{equation*}
w(f,\delta )=\sup\limits_{0<h\leq \delta ^{\frac{1}{2}}}\sup\limits_{x\in
\lbrack 0,1]}\mid f(x+h)-f(x)\mid .
\end{equation*}%

\begin{thm}
 Let $f\in [0,1]$ and $0<q<p\leq1$. Then for all $n\in
\mathbb{N}$, there exists an absolute constant $C>0$ such that
\begin{equation*}
\mid M_{n,p,q}(f;x)-f(x)\mid \leq C w_{2}(f,\delta _{n}(x)),
\end{equation*}%
where%
\begin{equation*}
\delta _{n}^{2}(x)=x^2 (p-1) + \frac{p^n}{[n+1]_{p,q}}x.\newline
\end{equation*}
\end{thm}

\parindent=0mm \textbf{Proof}. Let $g\in W_{2}.$ From Taylor's expansion, we
get \newline
\begin{equation*}
g(t)=g(x)+g^{\prime }(x)(t-x)+\int_{x}^{t}(t-u)~g^{\prime \prime
}(u)~du,~~~t\in \lbrack 0,A],~A>0,~~
\end{equation*}%
\newline
and by Lemma \ref{l2}(1), we get
\begin{equation*}
M_{n,p,q}(g;x)=g(x)+M_{n,p,q}\left( \int_{x}^{t}(t-u)~g^{\prime
\prime }(u)~du;x\right)
\end{equation*}%
\begin{equation*}
\mid M_{n,p,q}(g;x)-g(x)\mid \leq ~\biggl{|} M_{n,p,q}\left(
\int_{x}^{t}(t-u)~g^{\prime \prime }(u)~du;x\right) \biggl{|}
\end{equation*}
\begin{equation*}
\leq M_{n,p,q}\left( \mid \int_{x}^{t}\mid (t-u)\mid ~\mid g^{\prime
\prime }(u)\mid du\mid ;x\right)
\end{equation*}%
\begin{equation*}
\leq M_{n,p,q}\left( (t-x)^{2};x\right) \Vert g^{\prime \prime }\Vert .
\end{equation*}%
\newline
Using Lemma \ref{l2}(ii), we obtain
\begin{equation*}
\mid M_{n,p,q}(g;x)-g(x)\mid \leq x^2 (p-1)\Vert g^{\prime
\prime }\Vert +\frac{p^n}{[n+1]_{p,q}}x~~\Vert g^{\prime \prime }\Vert .
\end{equation*}%
\newline
On the other hand, by the definition of $M_{n,p,q}(f;x),$ we have
\begin{equation*}
\mid M_{n,p,q}(f;x)\mid \leq \Vert f\Vert .
\end{equation*}%
\newline
Now%
\begin{equation*}
\mid M_{n,p,q}(f;x)-f(x)\mid \leq \mid M_{n,p,q}\bigl{(}(f-g);x\bigl{)}-(f-g)(x)\mid +\mid M_{n,p,q}(g;x)-g(x)\mid
\end{equation*}%
\begin{equation*}
\leq \Vert f-g\Vert +x^2 (p-1)\Vert g^{\prime
\prime }\Vert + x \frac{p^n}{[n+1]_{p,q}} \Vert g^{\prime \prime }\Vert.
\end{equation*}%
Hence taking infimum on the right hand side over all $g\in W^{2}$, we get
\newline
\begin{equation*}
\mid M_{n,p,q}(f;x)-f(x)\mid \leq CK_{2}\bigl{(}f,\delta _{n}^{2}(x)\bigl{)}
\end{equation*}%
\newline
In view of the property of $K$-functional, we get

\begin{equation*}
\mid M_{n,p,q}(f;x)-f(x)\mid \leq Cw_{2}\bigl{(}f,\delta _{n}{(x)}\bigl{)}.
\end{equation*}

\parindent=8mm This completes the proof of the theorem. \newline

\begin{thm} If $f\in C[0,1]$, then
$$\bigl{|}M_{n,p,q}(f;x)-f(x)\bigl{|}\leq 2w_f\biggl{(}\sqrt{\frac{p^{n}}{[n+1]_{p,q}}}\biggl{)}$$ holds.
\end{thm}

\parindent=0mm \textbf{Proof}. Since $M_{n,p,q}(1,x)=1$, we have
\begin{eqnarray*}
    \bigl{|}M_{n,p,q}(f;x)-f(x)\bigl{|}&\leq& \frac{1}{p^{\frac{n(n+1)}{2}}}\sum\limits_{k=0}^{\infty}\left[
\begin{array}{c}
n+k \\
k%
\end{array}%
\right] _{p,q}x^{k} p^{-kn}P_{n-1}(x)\biggl{|}f\biggl{(}\frac{p^n~[k]_{p,q}}{[n+k]_{p,q}}\biggl{)}-f(x)\biggl{|}\\
&\leq&\biggl{\{} \frac{1}{p^{\frac{n(n+1)}{2}}} \sum\limits_{k=0}^{\infty}\left[
\begin{array}{c}
n+k \\
k%
\end{array}%
\right] _{p,q}x^{k} p^{-kn}P_{n-1}(x)\biggl{(}\frac{\bigl{|}\frac{p^n~[k]_{p,q}}{[n+k]_{p,q}}-x\bigl{|}^2}{\delta^2}+1\biggl{)}\biggl{\}}w_f(\delta) \\
   &=& \biggl{\{}\frac1{\delta^2} \frac{1}{p^{\frac{n(n+1)}{2}}} \sum\limits_{k=0}^{\infty}\left[
\begin{array}{c}
n+k \\
k%
\end{array}%
\right] _{p,q}x^{k}p^{-kn} P_{n-1}(x)\biggl{(}\frac{p^n[k]_{p,q}}{[n+k]_{p,q}}-x\biggl{)}^2\\
&&+\sum\limits_{k=0}^{\infty}\left[
\begin{array}{c}
n +k\\
k%
\end{array}%
\right] _{p,q}x^{k}p^{-kn}P_{n-1}(x)\biggl{\}}w_f(\delta)
\end{eqnarray*}
\begin{eqnarray*}
   &=& \bigl{\{}\frac1{\delta^2}\bigl{(}M_{n,p,q}(t^2;x)-2xM_{n,p,q}(t;x)+x^2M_{n,p,q}(1;x)\bigl{)}+1\bigl{\}}w_f(\delta) \\
&\leq& \biggl{\{}\frac1{\delta^2}\bigl{(}\frac{p^{n}}{[n+1]_{p,q}}\bigl{)}+1\biggl{\}}w_f(\delta).
\end{eqnarray*}
where $P_{n-1}(x)=\prod\limits_{s=0}^{n}(p^s-q^{s}x)$. Choosing $\delta=\delta_n=\sqrt{\frac{p^{n}}{[n+1]_{p,q}}}$, we have
$$\bigl{|}M_{n,p,q}(f;x)-f(x)\bigl{|}\leq2w_f(\delta_n).$$
\parindent=8mm This completes the proof of the theorem.\\

$~~~~~~~~~~$\parindent=8mm Now we give the rate of convergence of the operators $M_{n,p,q}$ in terms
of the elements of the usual Lipschitz class $Lip_M(\alpha)$.\\
$~~~~~~~~~~$Let $f\in C[0,1]$, $M>0$ and $0<\alpha\leq1$.We recall that $f$ belongs to the class
$Lip_M(\alpha)$ if the inequality
\begin{equation*}
    |f(t)-f(x)|\leq M|t-x|^\alpha~~~(t,x\in[0,1])
\end{equation*}
is satisfied.\\

\begin{thm} Let $0<q<p\leq1$. Then for each $f\in Lip_M(\alpha)$ we have
\begin{equation*}
    |M_{n,p,q}(f;x)-f(x)|\leq M\delta_n(x)^{\alpha}
\end{equation*}
where
\begin{equation*}
\delta _{n}^{2}(x)=\biggl{(} x^2(p-1)\biggl{)}+\frac{p^n}{{[n+1]}_{p,q}}x.\newline
\end{equation*}
\end{thm}

\parindent=0mm\textbf{Proof}. Let us denote $P_{n,k}(x)= \frac{1}{p^{\frac{n(n+1)}{2}}}\left[
\begin{array}{c}
n+k \\
k%
\end{array}%
\right] _{p,q}x^{k}p^{-kn} \prod\limits_{s=0}^{n}(p^s-q^{s}x)$. Then by the monotonicity of the operators $M_{n,p,q}$, we can write
\begin{eqnarray*}
 |M_{n,p,q}(f;x)-f(x)| &\leq& M_{n,p,q}\bigl{(}|f(t)-f(x)|;x\biggl{)} \\
  &\leq& \sum\limits_{k=0}^{\infty}P_{n,k}(x)\biggl{|}f\biggl{(}\frac{p^n~[k]_{p,q}}{[n+k]_{p,q}}\biggl{)}-f(x)\biggl{|} \\
   &\leq& M \sum\limits_{k=0}^{\infty}P_{n,k}(x)\biggl{|}\frac{p^n~[k]_{p,q}}{[n+k]_{p,q}}-x\biggl{|}^\alpha \\
   &=&M\sum\limits_{k=0}^{\infty}\biggl{(}P_{n,k}(x)\biggl{(}\frac{p^n~[k]_{p,q}}{[n+k]_{p,q}}-x\biggl{)}^2\biggl{)}^{\alpha/2}P^{\frac{2-\alpha}{2}}_{n,k}(x)
\end{eqnarray*}
Now applying the H\"{o}lder's inequality for the sum with $p=\frac2{\alpha}$ and $q=\frac2{2-\alpha}$ and taking into consideration Lemma \ref{l1}(i) and Lemma \ref{l2} (ii), we have
\begin{eqnarray*}
 |M_{n,p,q}(f;x)-f(x)| &\leq& M\biggl{(}\sum\limits_{k=0}^{\infty}P_{n,k}(x)\biggl{(}\frac{p^n~[k]_{p,q}}{[n+k]_{p,q}}-x\biggl{)}^2\biggl{)}^{\alpha/2}
  \biggl{(}\sum\limits_{k=0}^{\infty}P_{n,k}(x)\biggl{)}^{\frac{2-\alpha}{2}} \\
  &=&M\bigl{\{}M_{n,p,q}\bigl{(}(t-x)^2;x\bigl{)}\bigl{\}}^\frac{\alpha}{2}
\end{eqnarray*}
Choosing $\delta:\delta_{n}(x)=\sqrt{M_{n,p,q}\bigl{(}(t-x)^2;x\bigl{)}}$, we obtain
\begin{equation*}
   |M_{n,p,q}(f;x)-f(x)| \leq M\delta_n(x)^{\alpha}
\end{equation*}
Hence, the desired result is obtained.\\

\section{Statistical approximation properties}

The statistical version of Korovkin theorem for sequence of positive linear operators has been given by Gadjiev and Orhan \cite{go39}.

In this section, we give a statistical approximation theorem of the operator $M_{n,p, q}$.

Let $K$ be a subset of the set $\mathbb{N}$ of natural numbers. Then, the asymptotic density $\delta(K)$ of $K$ is defined as $\delta(K)=\lim_{n}\frac{1}{n}\big|\{k\leq n~:~k \in K\}\big|$ and $|.|$ represents the cardinality of the enclosed set. A sequence $x=(x_k)$ said to be statistically convergent to the number $L$ if for each $\varepsilon >0$, the set $K(\varepsilon)=\{k\leq n:|x_k-L|>\varepsilon\}$ has asymptotic density zero (see \cite{fast}), i.e.,
\begin{eqnarray*}\label{117}
\lim_{n}\frac{1}{n}\big|\{k\leq n:|x_k-L|\geq \varepsilon\}\big|=0.
\end{eqnarray*}
In this case, we write $st-\lim x =L$.

\begin{thm} Let $0<q_n<p_n\leq 1$ such that $\lim_n p_n=1$, and $\lim_n q_n=1$. Then for the operators $M_{n,p, q}$ satisfying the condition
\begin{eqnarray}\label{194}
st-\lim_n \|M_{n,p_n, q_n}(f_i(t); x)-{f_i(x)}\|_{C[0,1]}=0,~\textrm{for all}~i=0, 1, 2,
\end{eqnarray}
we have
\begin{eqnarray}\label{224}
st-\lim_n \|M_{n,p_n, q_n}(f(t); x)-{f(x)}\|_{C[0,1]}=0,~\textrm{for all}~f \in C[0, 1]
.\end{eqnarray}
\end{thm}

\begin{proof}  From Lemma \ref{l1}(i)-(ii), it is obvious that
\begin{eqnarray*}\label{117}
\lim_n \|M_{n,p_n, q_n}(1; x)-1\|_{C[0,1]}=0~~\text{and}~~\lim_n \|M_{n,p_n, q_n}(t; x)-x\|_{C[0,1]}=0,
\end{eqnarray*}
which yields
\begin{eqnarray*}\label{117}
st-\lim_n \|M_{n,p_n, q_n}(1; x)-1\|_{C[0,1]}=0~~\text{and}~~st-\lim_n \|M_{n,p_n, q_n}(t; x)-x\|_{C[0,1]}=0.
\end{eqnarray*}
Similarly, by Lemma \ref{l1}(iii), one can deduce that
\begin{eqnarray*}\label{117}
|M_{n, p_n, q_n}(t^2; x)-x^2|&\leq&\left|\frac{p_n^n}{[n+1]} x+p_n x^2-x^2\right|.
\end{eqnarray*}
Taking supremum over $x\in [0, 1]$ in the last inequality and by using $\lim_n p^n_n=a$ and $\lim_n q^n_n=b$ where $a,b \in (0, 1)$, we observe that
\begin{eqnarray*}\label{117}
\|M_{n, p_n, q_n}(t^2; x)-x^2\|_{C[0,1]}\leq\left|\frac{p_n^n}{[n+1]}+p_n-1 \right|
\end{eqnarray*}
which implies
\begin{eqnarray*}\label{117}
st-\lim_n \|M_{n,p_n, q_n}(t^2; x)-x^2\|_{C[0,1]}=0,~\textrm{for all}~i=0, 1, 2.
\end{eqnarray*}
It is noticed here that if we take $(p_n)$ and $(q_n)$ in the interval $(0, 1]$ such that $\lim _n p^n_n=a$, $\lim _n q^n_n=b$ where $a,b \in (0, 1)$, we guarantee that $\lim_n \frac{1}{[n]_{p_n,q_n}}=0$. For example, if we choose  $(q_n)=(1-\frac{1}{n})<(e^{1/2n}(1-\frac{1}{n}))=(p_n)$ then $\lim_n q^n_n=e^{-1}$ and $\lim_n p^n_n=1/\sqrt{e}$. In this case we guarantee that $\lim_n \frac{1}{[n]_{p_n,q_n}}=0$. Besides
if we take the sequences $q_n \in (0, 1)$ and $p_n\in (q_n, 1]$ such that $\lim _n p_n=1$ and $\lim _n q_n=1$. Thus,
$\lim_n \frac{1}{[n]_{p_n,q_n}}=0$.

Now, let us take $f \in C[0, 1]$ and $x \in [0, 1]$ be fixed. Since $f$ is bounded on the whole plane there exists a number $M$ such that $|f(x)|\leq M$ for all
$-\infty<x<\infty$.
Thus, we may write
\begin{eqnarray}\label{u3}
|f(t)-f(x)|\leq 2M,~\,\,-\infty<t, x<\infty.
\end{eqnarray}
From the continuity of $f$ at the point $x\in  [0, 1]$, for given $\varepsilon >0$, there exists a number $\delta=\delta(\varepsilon)>0$ such that
\begin{eqnarray}\label{u4}
|f(t)-f(x)|<\varepsilon~~\textrm{whenever}\,\,\forall~|t-x|<\delta.
\end{eqnarray}
Setting $\varphi(x)=(t-x)^2$ and hence we derive the consequence from (\ref{u3}) and (\ref{u4}) that
\begin{eqnarray*}
|f(t)-f(x)|< \varepsilon+\frac{2 M}{\delta^2}\varphi(x).
\end{eqnarray*}
By taking into account the linearity and positivity of $M_{n,p_n, q_n}$, we obtain that
\begin{eqnarray*}
&&|M_{n,p_n, q_n}(f(t); x)-{f(x)}|\\&\leq&M_{n,p_n, q_n}\big(|f(t)-f(x)|; x\big)+f(x)~|M_{n,p_n, q_n}(1; x)-1|\\
&\leq&\left|M_{n,p_n, q_n}\left(\varepsilon+\frac{2M}{\delta^2}\varphi(x); x\right)\right|+N~|M_{n,p_n, q_n}(1; x)-1|\\
&\leq&\varepsilon+\left(\varepsilon+M+
\frac{2M}{\delta^2}a^2\right)
|M_{n,p_n, q_n}(1; x)-1|+\frac{4M}{\delta^2}a~\bigg|M_{n,p_n, q_n}(t; x)-x\bigg|
\\
&+&\frac{2M}{\delta^2}~\bigg|M_{n,p_n, q_n}\big(t^2; x\big)-x^2\bigg|
\end{eqnarray*}
where $a=\max\{|x|\}$ and $N=\|f\|_{C[0,1]}$. Then taking supremum over $x\in [0, 1]$ in the last inequality, we have
\begin{eqnarray*}
\|M_{n,p_n, q_n}(f(t); x)-{f(x)}\|_{C[0,1]}\leq \varepsilon+\eta \sum_{i=0}^{2}\|M_{n,p_n, q_n}(f_{i}(t); x)-f_i(x)\|_{C[0,1]}
\end{eqnarray*}
where $\eta=\max\left\{\varepsilon+M+
\frac{2M}{\delta^2}a^2, \frac{4M}{\delta^2}a,\frac{2M}{\delta^2}\right\}$. Now, for a given $\varepsilon' >0$, choose a number $\varepsilon >0$ such that $\varepsilon <\varepsilon'$. Then, the following sets can be defined as
\begin{eqnarray*}
\mathcal{A}:&=&\left\{n \in \mathbb N:\big\|
M_{n,p_n, q_n}(f(t); x)-{f(x)}\|_{C[0,1]}\geq\varepsilon'\right\},\\
\mathcal{A}_{i}:&=&\left\{n \in \mathbb N:\|M_{n,p_n, q_n}(f_i(t); x)-{f_i(x)}\|_{C[0,1]}\ge \frac{\varepsilon' -\varepsilon}{3N}\right\},\,\,i=0,1,2.
\end{eqnarray*}
Then, $\mathcal{A}\subset \mathcal{A}_{0}\cup \mathcal{A}_{1}\cup
\mathcal{A}_{2}$. By using the facts that $st-\lim_n p_n=1$, $st-\lim_n q_n=1$ and (\ref{194}) we immediately conclude that \begin{eqnarray*}\label{117}
st-\lim_n \|M_{n,p_n, q_n}(f(t); x)-{f(x)}\|_{C[0,1]}=0~\textrm{for all}~f \in C[0, 1],\end{eqnarray*} which completes the proof.
\end{proof}

\subsection{Rate of Statistical Convergence}

In this part, rates of statistical convergence of the operator (\ref{87}) by means of modulus of continuity \cite{lb313} are introduced.\\
The modulus of continuity for the space $f \in C[0,1]$ is defined by
$$\tilde{W}(f,\delta)= \sup \limits_{x,t\in [0,1], |t-x|\leq\delta}|f(t)-f(x)|$$
where $\tilde{W}(f,\delta)$ satisfies the following conditions:\\
$\forall f\in C[0,1]$

\begin{enumerate}
\item  [(i)]$\lim_{\delta \rightarrow 0}\tilde{W}(f,\delta)=0$\\
\item [(ii)]$|f(t)-f(x)|\leq \tilde{W}(f,\delta)\big(\frac{|t-x|}{\delta}+1\big)$
\end{enumerate}

\begin{thm}
Let the sequences $p=(p_n)$, $q=(q_n)$, $0<q_n<p_n\leq 1$ satisfies the condition $p_n \rightarrow 1$, $q_n \rightarrow 1$, so we have
$$|M_{n, p_n, q_n}(f;x)-f(x)|\leq 2 \tilde{W}\big(f;\sqrt{\delta_{n}(x)}\big)$$ where $\delta_n(x)=\frac{{p_n}^n}{[n+1]_{p_n,q_n}}x+(p_n-1)x^2$
\end{thm}

Proof is similar to the technique used in previous theorems.

\newpage

\section{Example}

With the help of Matlab, we show comparisons and some illustrative
graphics \cite{ma1}  for the convergence of operators (\ref{87}) to the function $$f(x)=(x-\frac{1}{3})(x-\frac{1}{2})(x-\frac{3}{4})$$ under different parameters. Because of our machines have not enough speed and power to compute the complicated infinite series, we have to investigate our series only for finite sum.

From figure $(\ref{f1})$, we can observe that as the value of $k$ increases, $(p,q)$-Meyer-Konig and Zeller operators $M_{n,p, q}(1,x)$ given by (\ref{ee1}) converges towards $1$. In figure $(\ref{f1})$ (a), value of $k=0 ~\text{to} ~100$ whereas in figure  $(\ref{f1})$ (b), value of $k=0 ~\text{to} ~500$ for $n=3.$

Also, from figure $(\ref{f2})$, it can be observed
that as the value of $p, q$ approaches towards $1$ provided $ 0 < q < p \leq 1,$ operators converge towards the function.


\begin{figure}[ht]\label{f1}
    \centering
    \subfigure[]
    {
        \includegraphics[height=4cm, width=6cm]{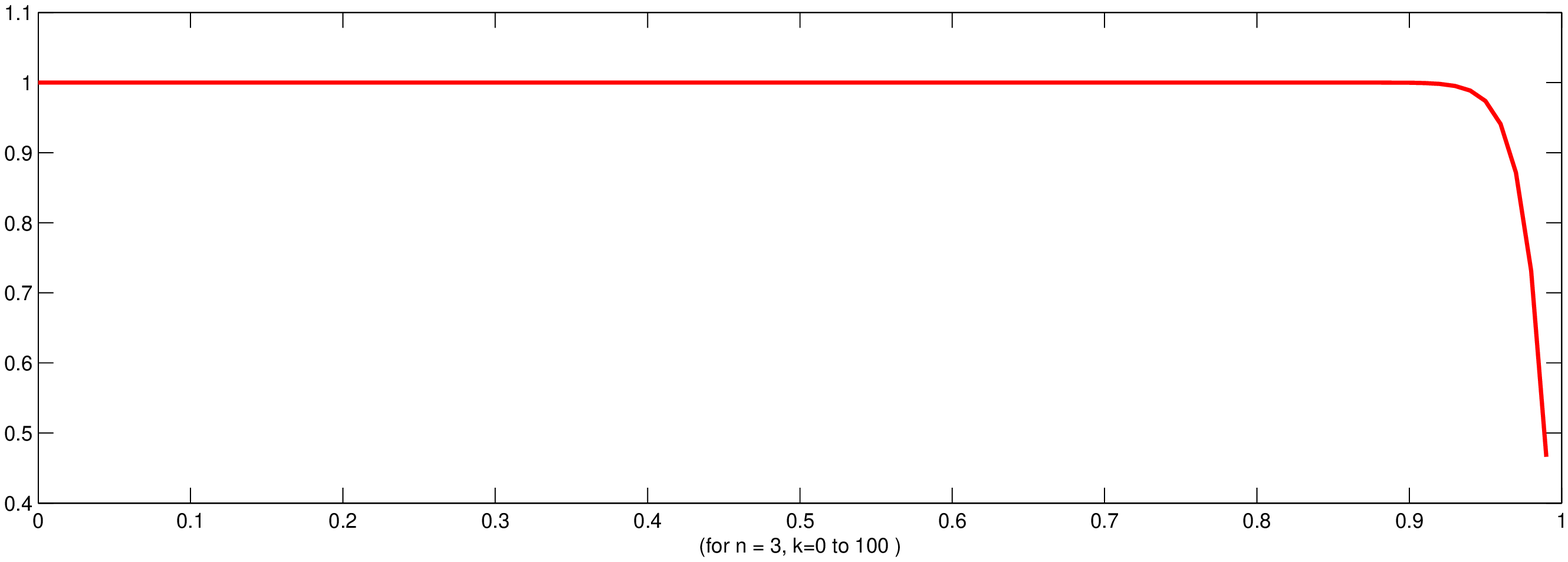}
        \label{fig:first_sub}
    }
    \subfigure[]
    {
        \includegraphics[height=4cm, width=6cm]{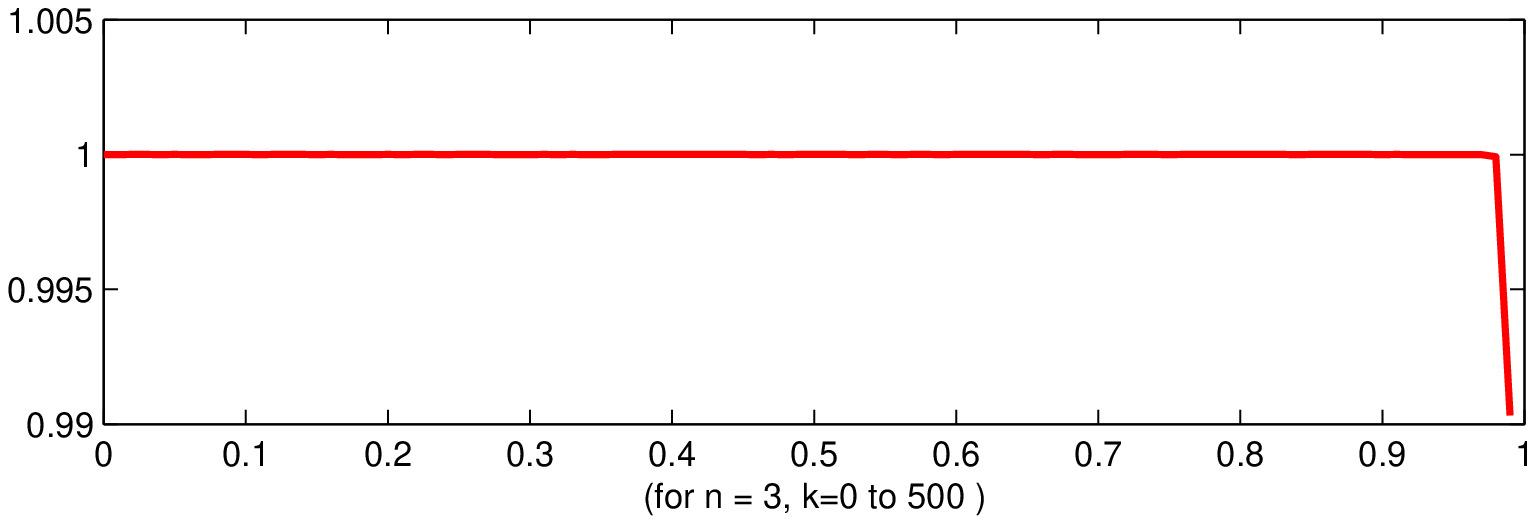}
        \label{fig:second_sub}
    }
    \caption{$(p,q)$-Meyer-Konig and Zeller operator}\label{f1}
\end{figure}

\begin{figure}[ht]\label{f2}
    \centering
    \subfigure[]
    {
        \includegraphics[height=4cm, width=6cm]{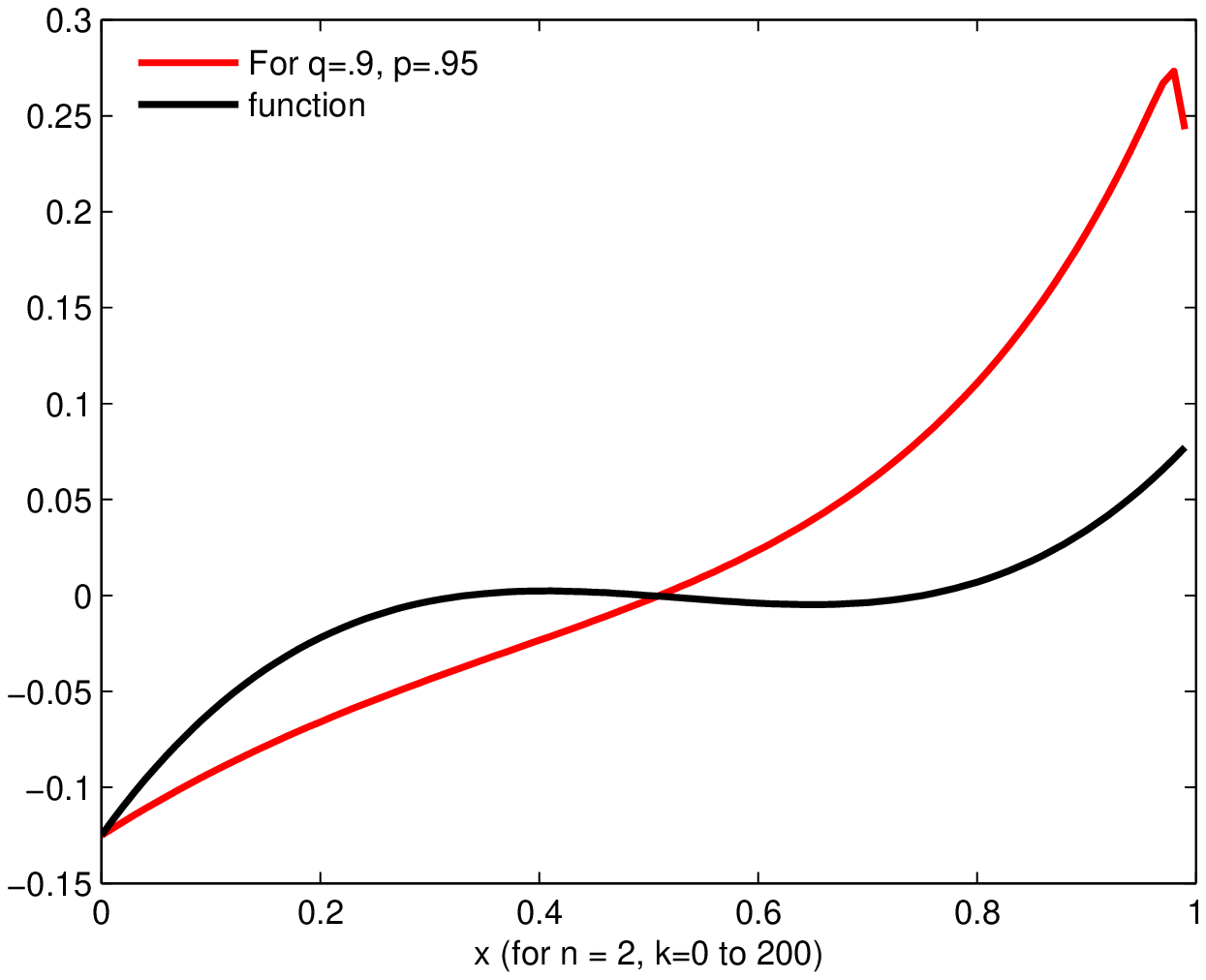}
        \label{fig:first_sub}
    }
    \subfigure[]
    {
        \includegraphics[height=4cm, width=6cm]{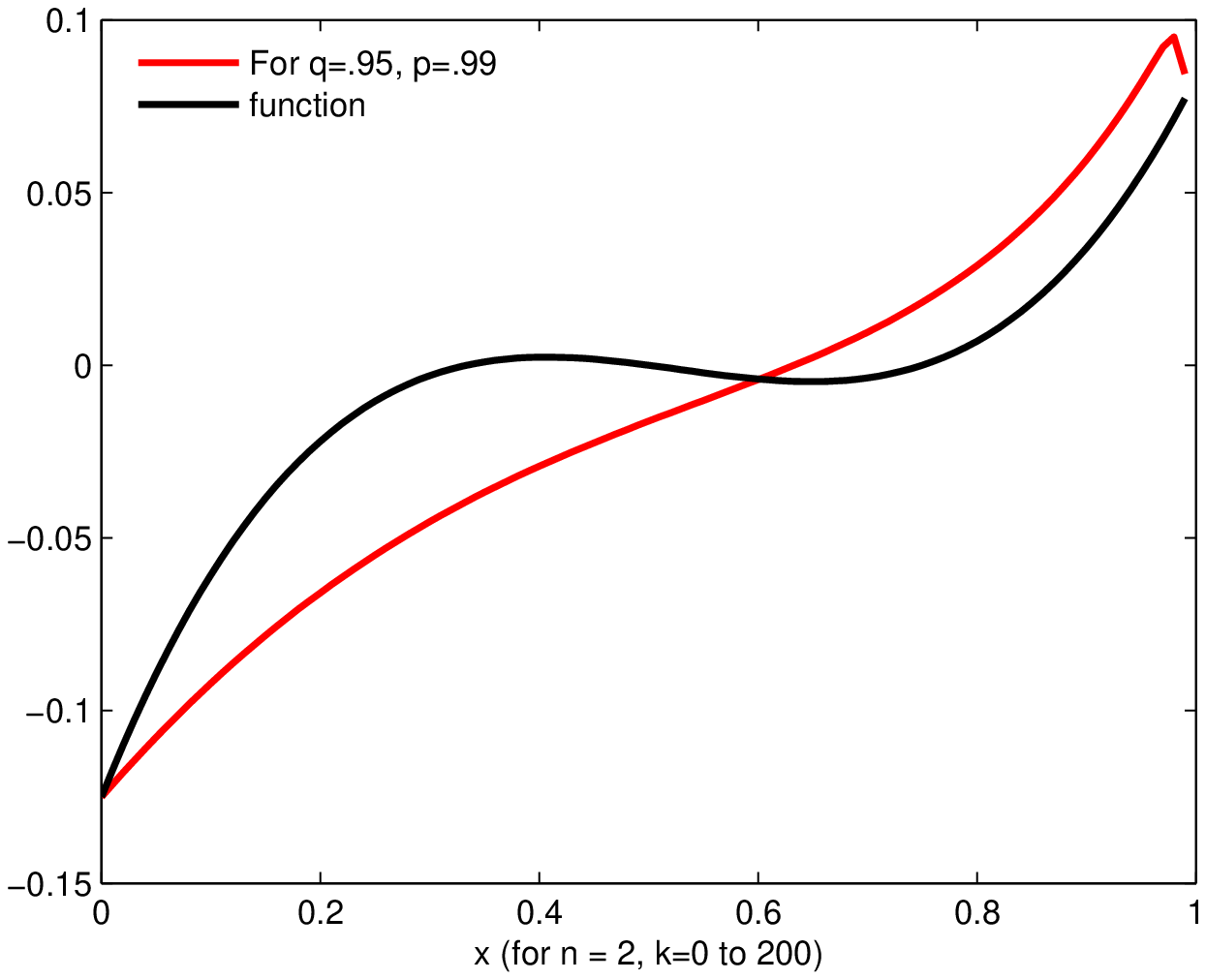}
        \label{fig:second_sub}
    }
    \subfigure[]
    {
        \includegraphics[height=4cm, width=6cm]{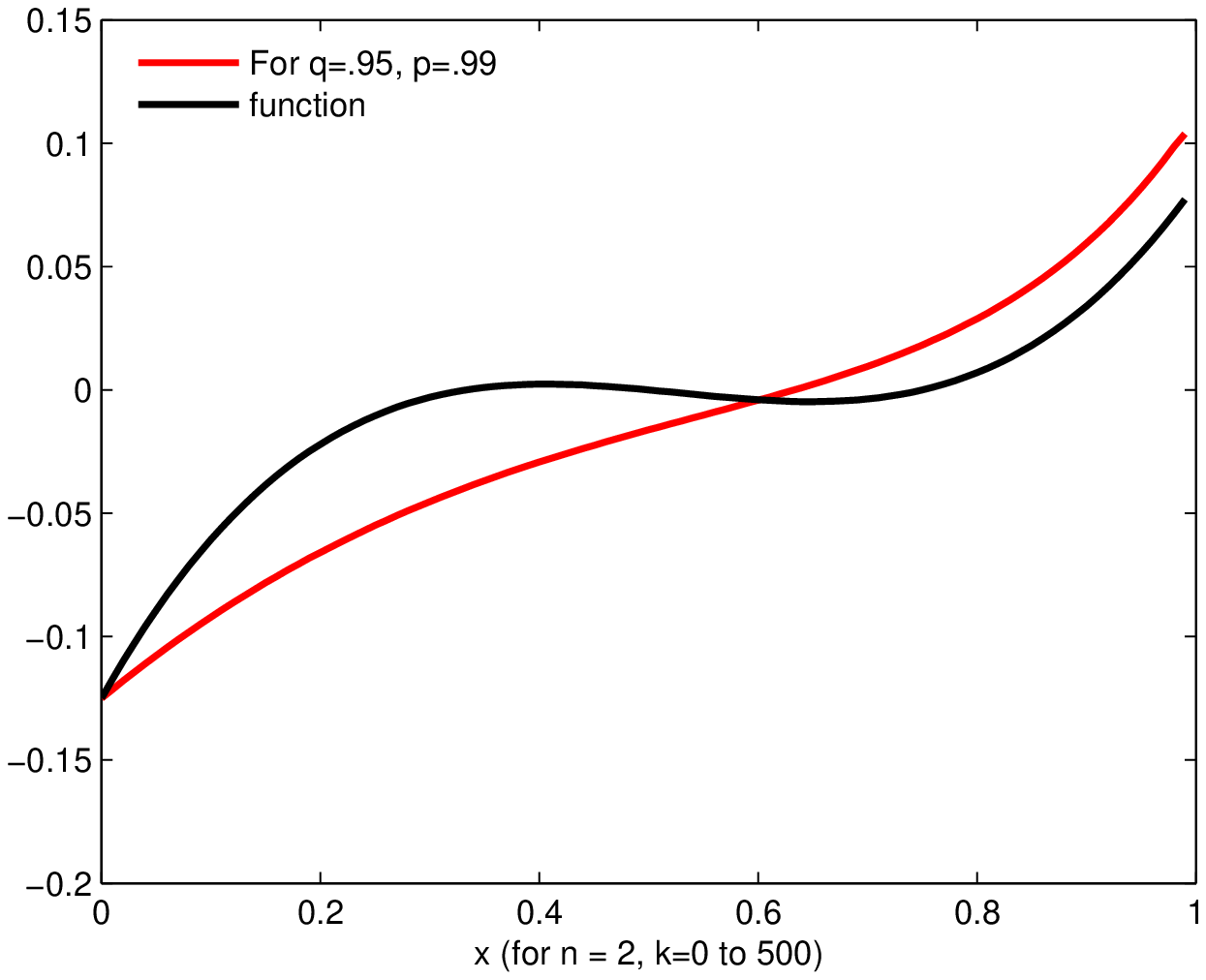}
        \label{fig:third_sub}
    }
    \caption{$(p,q)$-Meyer-Konig and Zeller operators}\label{f2}
    \label{fig:sample_subfigures}
\end{figure}

%
%

\newpage

\section{ Future work}

Currently, we are working on Kantorovich form of $(p,q)$-Meyer-Konig and Zeller operators.


\end{document}